\documentclass[a4paper,reqno,12pt]{amsart}

\usepackage{amssymb,eucal}

\newtheorem{theorem}{Theorem}[section]
\newtheorem{lemma}[theorem]{Lemma}
\newtheorem{corollary}[theorem]{Corollary}
\newtheorem{proposition}[theorem]{Proposition}

\numberwithin{equation}{section}

\theoremstyle{definition}
\newtheorem{definition}[theorem]{Definition}
\newtheorem*{example*}{Example}
\newtheorem{example}[theorem]{Example}

\newtheorem{remark}[theorem]{Remark}
\newtheorem*{remark*}{Remark}

\newcommand\frsl{\mathfrak{sl}}

\newcommand\frso{\mathfrak{so}}
\newcommand\frsp{\mathfrak{sp}}
\newcommand\frsu{\mathfrak{su}}

\newcommand{\calA}{\mathcal{A}}

\newcommand{\calJ}{\mathcal{J}}
\newcommand{\calT}{\mathcal{T}}

\newcommand{\frg}{\mathfrak{g}}
\newcommand{\fra}{\mathfrak{a}}
\newcommand{\frb}{\mathfrak{b}}
\newcommand{\frc}{\mathfrak{c}}
\newcommand{\frs}{\mathfrak{s}}
\newcommand{\frd}{\mathfrak{d}}

\newcommand{\fri}{\mathfrak{i}}

\newcommand{\frm}{\mathfrak{m}}
\newcommand{\fru}{\mathfrak{u}}
\newcommand{\frv}{\mathfrak{v}}
\newcommand{\frz}{\mathfrak{z}}
\newcommand{\frr}{\mathfrak{r}}

\newcommand{\bF}{\mathbb{F}}
\newcommand{\bC}{\mathbb{C}}
\newcommand{\bR}{\mathbb{R}}

\newcommand{\bZ}{\mathbb{Z}}

\newcommand{\bN}{\mathbb{N}}

\newcommand{\id}{{\mathrm{id}}} 
\newcommand{\bi}{\mathbf{i}}

\DeclareMathOperator{\Der}{\mathrm{Der}}

\DeclareMathOperator{\End}{\mathrm{End}}
\DeclareMathOperator{\Hom}{\mathrm{Hom}}

\DeclareMathOperator{\trace}{trace}

\begin{document}

\title[Quadratic Lie algebras of low dimension]{Classification of quadratic Lie algebras of low dimension}

\author[Sa\"id Benayadi]{Sa\"id Benayadi}

\address{Universit\'e de Lorraine, IECL, CNRS UMR 7502, Ile
du Saulcy, F-57045 Metz cedex, France} \email{said.benayadi@univ-lorraine.fr}

\author[Alberto Elduque]{Alberto Elduque$^{\star}$}
\thanks{$^{\star}$ Supported by the Spanish
Ministerio de Econom\'{\i}a y Competitividad and FEDER (MTM2010-18370-C04-02) and by
the Diputaci\'on General de Arag\'on---Fondo Social Europeo (Grupo de Investigaci\'on de \'Algebra)}
\address{Departamento de Matem\'aticas e
Instituto Universitario de Matem\'aticas y Aplicaciones, Universidad de
Zaragoza, 50009 Zaragoza, Spain} \email{elduque@unizar.es}

\begin{abstract}
In this paper we give the classification of the irreducible non solvable Lie algebras of
dimensions $\leq 13$ with nondegenerate, symmetric and
invariant bilinear forms.
\end{abstract}



\date{April 21, 2014}

\maketitle


\section{Introduction}

In this work, we consider finite-dimensional Lie algebras over a field $\bF$ of characteristic $0$. These restrictions will be assumed throughout the paper.

A Lie algebra $\frg$ is said to be \emph{quadratic} (some authors call these \emph{quasi-classical}) if it is endowed with a symmetric nondegenerate invariant bilinear form $B$. The invariance means $B([x,y],z)=B(x,[y,z])$ for any $x,y,z\in \frg$.
Quadratic Lie algebras appear in many areas of Mathematics and Physics. In particular, the importance of quadratic Lie algebras in
conformal field theory has to do with following fact \cite{FS}: quadratic Lie algebras are precisely
the Lie algebras for which a Sugawara construction exists.

To study the structure of quadratic Lie algebras,
A. Medina and Ph. Revoy introduced in \cite{m-r85} the concept of \emph{double extension}. Using this notion, they gave
an inductive description of these algebras. Another interesting concept, the \emph{$T^\star$-extensions}, was given
by M.~Bordemann in \cite{borde}. These two concepts are very useful tools to construct examples and to study the structure of quadratic Lie algebras. But, even with the use of these concepts,
the classification (up isomorphism) of quadratic Lie algebras of dimension $\geq 7$ is not easy.
In this paper, we give the classification of the non solvable irreducible   quadratic Lie algebras of
dimensions $\leq 13$, both over algebraically closed fields and over the real field To obtain this classification, we show that, assuming $\bF$ is algebraically closed, the Levi component of a non solvable irreducible
quadratic Lie algebra is $\frsl_2(\bF)$ in most cases, and then we use the well-known representation theory of $\frsl_2(\bF)$.

Recall that the classification of the nilpotent quadratic Lie algebras of dimensions
$\leq 7$ is obtained in \cite{FavS}, the classification of the real non solvable irreducible quadratic Lie algebras of
dimensions $\leq 9$ and the classification of the real solvable  quadratic Lie algebras of
dimensions $\leq 6$ are obtained in \cite{Otto}.

In Section \ref{se:preliminaries} the main definitions and some useful results are recalled. Any quadrat6ic Lie algebra is an orthogonal sum of \emph{irreducible} quadratic Lie algebras, so it is enough to classify these.
Then in Section \ref{se:closed} the classification of the nonsolvable irreducible quadratic Lie algebras of dimension $\leq 13$ over an algebraically closed field is given, while Section \ref{se:real} is devoted to the real case.

\medskip

\section{Definitions and preliminary results}\label{se:preliminaries}

\begin{definition}
\begin{enumerate}

\item A bilinear form $B$ on a Lie algebra $(\frg,[\,,\,])$ is said to be \emph{invariant} if
\[
B([x,y],z)= B(x,[y,z]),
\]
for any $x,y,z  \in {\frg}$.

\item Let $({\frg},[\,,\,])$ be a Lie algebra, ${\frg}$ is called \emph{quadratic} if it admits a nondegenerate, symmetric and invariant bilinear form $B$.
In this case, $B$ is called an \emph{invariant scalar product} on ${\frg}$.

\item Let ${\frg}$ be a quadratic Lie algebra and let $B$ be an invariant scalar product on  ${\frg}$. An ideal $\fri$ of ${\frg}$ is called \emph{$B$-nondegenerate} (or just \emph{nondegenerate}) if $B_{\vert_{\fri\times \fri}}$ is nondegenerate.

\item Let ${\frg}$ be a quadratic Lie algebra and let $B$ be an invariant scalar product on  ${\frg}$. We say that ${\frg}$  is \emph{$B$--irreducible} (or just \emph{irreducible}) if $\frg$ contains no non-trivial $B$-nondegenerate ideals.
\end{enumerate}
\end{definition}

The study of quadratic Lie algebras is reduced to those having no non-trivial nondegenerate ideals. Indeed, if
we consider a quadratic Lie algebra $\frg$ with an invariant scalar product $B$, and
$\fri$ is an ideal of $\frg$, then $\fri^\perp$ is an ideal of $\frg$ because $B$ is invariant
(Here $\fri^\perp$ denotes the orthogonal subspace to $\fri$ relative to $B$).
Now  if $\fri$ is a $B$-nondegenerate ideal, then $\frg=\fri\oplus\fri^\perp$, and both
ideals: $\fri$ and $\fri^\perp$, are quadratic Lie algebras. Hence $\frg$ is easily seen to be an
orthogonal direct sum of irreducible quadratic Lie algebras.

\begin{proposition}[\cite{m-r85}]\label{LEM}
Let $({\frg},B)$ be a quadratic Lie algebra.
Then:
\begin{enumerate}
\item $[{\frg},{\frg}]^\perp= {\frz}({\frg})$,  where ${\frz}({\frg})$ is the center of ${\frg}$.

\item If $\fri$ is a centerless ideal of  ${\frg},$ then $\fri$ is a nondegenerate ideal of $\frg$.
\end{enumerate}
\end{proposition}
\begin{proof}
Note that $x\in [\frg,\frg]^\perp$ if and only if $B(x,[\frg,\frg])=0$, if and only if $B([x,\frg],\frg)=0$, if and only if (since $B$ is nondegenerate) $[x,\frg]=0$.

For the second part note that for any ideal $\fri$, $B([\fri,\fri^\perp],\frg)\subseteq B(\fri,\fri^\perp)=0$, so $[\fri,\fri^\perp]=0$ and hence $\fri\cap\fri^\perp$ is contained in the center of $\fri$.
\end{proof}

This Proposition proves that if ${\frg}$ is a quadratic Lie algebra, then ${\frg}$ is perfect (i.e., $[{\frg},{\frg}]= {\frg}$) if and only if the center of $\frg$
is trivial.

Now, let us recall the concept of \emph{double extension} in the case of quadratic Lie algebras (see \cite{m-r85}).

Let $({\frg}_1,[\,,\,]_1,B_1)$ be a quadratic Lie algebra and let $({\frg}_2,[\,,\,]_2)$ be a Lie algebra
(not necessarily quadratic) such that there exists a homomorphism of Lie algebras
$\varphi: {\frg}_2 \rightarrow \Der({\frg}_1,B_1)$, where
$\Der({\frg}_1,B_1)$ denotes the Lie algebra of the derivations of $\frg$ wich are skew-symmetric relative to $B_1$ (i.e., $\Der(\frg,B_1)=\Der(\frg)\cap\frso(\frg,B_1)$).
Since we have $\varphi({\frg}_2)\subseteq \Der({\frg}_1,B_1)$,
the bilinear map $\psi: {\frg}_1\times {\frg}_1 \rightarrow ({\frg}_2)^*$
given by $\psi(x_1,y_1): z_2\mapsto B_1(\varphi(z_2)(x_1),y_1)$
for any $x_1,y_1\in \frg_1$ and $z_2\in \frg_2$, is a $2$-cocycle, where $({\frg}_2)^*$ is considered as a trivial ${\frg}_1$-module. Consequently,  the vector space ${\frg}_1\oplus ({\frg}_2)^*$, endowed with the multiplication:
\[
[x_1+f,y_1+h]_c:=  [x_1,y_1]_1+\psi(x_1,y_1),
\]
for any $x_1,y_1 \in {\frg}_1$ and  $f,h\in ({\frg}_2)^*$,
is a Lie algebra. This Lie algebra is the central extension  of ${\frg}_1$ by means of $\psi$.

Let $\pi$ be the co-adjoint representation of  ${\frg}_2$ For $x_2 \in {\frg}_2$, an easy computation proves that
the endomorphism  ${\bar  \varphi}(x_2)$ of ${\frg}_1\oplus ({\frg}_2)^*$  defined by: ${\bar  \varphi}(x_2)(x_1+f):=
\varphi(x_2)(x_1)+\pi(x_2)(f)$, for any $x_1 \in {\frg}_1, f\in ({\frg}_2)^*$, is a derivation of the Lie algebra
$({\frg}_1\oplus ({\frg}_2)^*, [\,,\,]_c)$. Next,
it is easy to see that the linear map ${\bar  \varphi}: {\frg}_2 \rightarrow \Der({\frg}_1\oplus ({\frg}_2)^*)$ is
a homomorphism of Lie algebras. Therefore, one can consider
${\frg}:= {\frg}_2 \ltimes_{\bar  \varphi}({\frg}_1\oplus ({\frg}_2)^*)$,
the semi-direct product of  ${\frg}_1\oplus ({\frg}_2)^*$ by ${\frg}_2$ by means of ${\bar  \varphi}$.
As a vector space ${\frg}= {\frg}_2\oplus  {\frg}_1\oplus ({\frg}_2)^*$, and the bracket
of the Lie algebra ${\frg}$ is given by:
\begin{multline*}
[x_2+x_1+f,y_2+y_1+h]\\ =[x_2,y_2]_2 + \Big([x_1,y_1]_1 + \varphi(x_2)(y_1) - \varphi(y_2)(x_1)\Big)\\
 + \Big(\pi(x_2)(h) - \pi(y_2)(f) + \psi(x_1,y_1)\Big),
\end{multline*}
for any $(x_2,x_1,f), (y_2,y_1,h)  \in {\frg}_2 \oplus {\frg}_1 \oplus ({\frg}_2)^*$. Moreover, if $\gamma: {\frg}_2\times {\frg}_2 \rightarrow {\bF}$ is any invariant, symmetric bilinear form on ${\frg}_2$ (we may take $\gamma=0$), it is easy to see that the bilinear form $B_{\gamma}: {\frg}\times {\frg} \rightarrow {\bF}$ defined by:
\begin{multline}\label{eq:Bgamma}
B_{\gamma}(x_2+x_1+f,y_2+y_1+h):=\\
 \gamma(x_2,y_2) + B(x_1,y_1) + f(y_2) + h(x_2),
\end{multline}
for any $(x_2,x_1,f), (y_2,y_1,h)  \in {\frg}_2 \oplus {\frg}_1 \oplus ({\frg}_2)^*,$  is an invariant scalar product on ${\frg}$. In this situation, ${\frg}$
(or $({\frg},B_0))$ is called the \emph{double extension} of $({\frg}_1,[\,,\,]_1,B_1)$ by ${\frg}_2$ by means of $\varphi$. We will denote this Lie algebra by $T_\varphi(\frg_1,B_1,\frg_2)$.

If ${\frg}_1= \{0\}$, the double extension ${\frg}_2 \oplus ({\frg}_2)^*$ is called also the \emph{trivial $T^\star$-extension} of the Lie algebra ${\frg}_2$ and it
is denoted by $T_0^\star({\frg}_2)$.

\begin{remark}\label{re:double_extension}
Let $(\frg_1,[\,,\,],B_1)$ be a quadratic Lie algebra, let $(\frg_2,[\,,\,]_2)$ be a Lie algebra and let $\varphi:\frg_2\rightarrow \Der(\frg_1,B_1)$ be a Lie algebra homomorphism. Then for any $0\ne\alpha\in\bF$, the linear map
\[
\begin{split}
T_\varphi(\frg_1,B_1,\frg_2)&\longrightarrow T_\varphi(\frg_1,\alpha B_1,\frg_2)\\
x_2+x_1+f\ &\mapsto\quad x_2+x_1+\alpha f,
\end{split}
\]
for any $x_2\in\frg_2$, $x_1\in\frg_1$ and $f\in\frg_2^*$, is easily checked to be a Lie algebra isomorphism. Hence we may always scale the bilinear form $B_1$ and get isomorphic double extensions.

Moreover, if $\frg_1=\frg_1^1\oplus\frg_1^2$ (direct sum of ideals) with $\frg_1^2$ abelian, $B_1$ is the orthogonal sum $B_1=B_1^1\perp B_1^2$ of the nondegenerate bilinear forms $B_1^i$ in $\frg_1^i$, $i=1,2$, and $\varphi$ leaves both $\frg_1^1$ and $\frg_1^2$ invariant, then the linear map
\[
\begin{split}
T_\varphi(\frg_1,B_1,\frg_2)&\longrightarrow T_\varphi(\frg_1,B_1^1\perp \alpha^2 B_1^2,\frg_2)\\
x_2+(u_1+u_2)+f\ &\mapsto\quad x_2+(u_1+\alpha^{-1}u_2)+f
\end{split}
\]
is again a Lie algebra isomorphism. Hence over algebraically closed fields we may scale `partially' the bilinear form on $\frg_1$ and get isomorphic double extensions, and over the real field we may scale by a positive scalar to get isomorphic double extensions. \qed
\end{remark}

\begin{proposition}\label{L1}
Let $({\frg}= {\frs}\oplus {\frr},B)$ be a non solvable and non
simple irreducible quadratic Lie algebra, where ${\frs}$ is a
Levi subalgebra of $\frg$ and ${\frr}$ is the solvable radical of $\frg$.
Then:
\begin{enumerate}
\item ${\frr}^\perp \subseteq {\frz}({\frr})$,
\item the ${\frs}$-module ${\frr}^{\perp}$
is isomorphic to the adjoint module of $\frs$.
\end{enumerate}
\end{proposition}

\begin{proof}
The subspace ${\fri}:= \{x \in {\frs} : B(x,{\frr})= 0\}$ is an ideal of
${\frs}$, because $B$ is invariant. Since $B([{\fri},{\frr}],{\frg})=
B({\fri},[{\frr},{\frg}])= 0$, then $[{\fri},{\frr}]= 0$ and hence ${\fri}$ is a semisimple ideal of $\frg$ and it is nondegenerate by Proposition \ref{L1}. The fact that
$\frg$ is irreducible and non simple implies ${\fri}= \{0\}$. It follows that
${\frr}^\perp$ is contained in ${\frr}$. Since $B([{\frr}^\perp,{\frr}], {\frg})= B({\frr}^\perp,[{\frr}, {\frg}])= 0$,
then we get $[{\frr}^\perp,{\frr}]= 0$, so ${\frr}^\perp \subseteq {\frz}({\frr})$.

Let us consider  the map $\varphi: {\frr}^\perp \rightarrow {\frs}^*$, defined by:
$\varphi(x):= B(x,.)$, for any $x\in {\frr}^\perp$. The non-degeneracy (resp. invariance)
of $B$ implies that $\varphi$ is injective (resp. a homomorphism of ${\frs}$-modules). Since the dimension of $\frr^\perp$ equals the dimension of $\frs$,  it follows that $\varphi$ is an isomorphism of ${\frs}$-modules. Also, $\frs$ is semisimple, so its Killing form is nondegenerate and hence the $\frs$-modules $\frs$ and $\frs^*$ are isomorphic.
\end{proof}

\begin{corollary}\label{co:trivialTextension}
Let $({\frg}= {\frs}\oplus {\frr},B)$ be a non solvable and non
simple irreducible quadratic Lie algebra, where ${\frs}$ is a
Levi subalgebra of $\frg$ and ${\frr}$ is the solvable radical of $\frg$. Assume that $\frr=\frr^\perp$. Then $\frs$ is simple and $\frg$ is isomorphic to the trivial $T^\star$ extension $T_0^\star(\frs)$.
\end{corollary}
\begin{proof}
By Proposition \ref{L1}, $\frr$ is abelian and isomorphic to the adjoint module $\frs\simeq \frs^*$ for $\frs$. Hence $\frg$ is isomorphic to $T_0^\star(\frs)$. The irreducibility of $\frg$ forces $\frs$ to be simple.
\end{proof}

\begin{corollary}\label{L2}
Let $({\frg}= {\frs}\oplus {\frr},B)$ be a non solvable and non simple irreducible  quadratic Lie algebra, where ${\frs}$ is a
Levi subalgebra and ${\frr}$ is the solvable radical of $\frg$. Let $m$ be
the number of the simple ideals of ${\frs}$ and $n$ the number of irreducible submodules in the decomposition of
the ${\frs}$-module ${\frr}$ as a direct sum of simple ${\frs}$-submodules.
Then $m\leq n$.
\end{corollary}

\begin{proof} By Proposition \ref{L1}, the
${\frs}$-modules ${\frs}$ and ${\frr}^{\perp}$ are isomorphic. Since ${\frr}^{\perp}$ is a submodule of the
${\frs}$-module ${\frr},$ then $m\leq n$.
\end{proof}

\begin{proposition}\label{pr:R_Rperp}
Let $({\frg}= {\frs}\oplus {\frr},B)$ be a non solvable and non simple irreducible quadratic Lie algebra, where ${\frs}$ is a
Levi subalgebra of $\frg$ and ${\frr}$ is the solvable  radical of $\frg$.

If ${\frs}$ is a simple Lie algebra, then  either
\[
{\frr}= {\frr}^{\perp}\quad \text{or}\quad {\frr}^{\perp}\subseteq [{\frr},{\frr}].
\]
\end{proposition}

\begin{proof}
Since the ${\frs}$-modules ${\frs} $ and ${\frr}^\perp$ are isomorphic, the ${\frs}$-module
${\frr}^{\perp}$ is irreducible. Therefore, either ${\frr}^{\perp}\cap [{\frr},{\frr}]= 0$  or
${\frr}^{\perp}\cap [{\frr},{\frr}]={\frr}^{\perp}$.

We know, by Proposition \ref{L1}, that ${\frr}^{\perp}\subseteq {\frr}$. If
we assume that ${\frr}^{\perp}\cap [{\frr},{\frr}]= 0$,  then there exists a $\frs$-submodule $\frm$  of ${\frr}$ such that $[{\frr},{\frr}]\subseteq \frm$  and
${\frr}= {\frr}^{\perp}\oplus \frm$. Then $\frm$ is a nondegenerate ideal of $\frg$.
Since $\frg$ is irreducible and not solvable, we conclude that $\frm= 0$, so ${\frr}= {\frr}^{\perp}$.

Finally, if ${\frr}^{\perp}\cap [{\frr},{\frr}]={\frr}^{\perp},$ then ${\frr}^{\perp}\subseteq
[{\frr},{\frr}]$.
\end{proof}


\section{Classification of the non solvable irreducible   quadratic Lie algebras of dimension $\leq 13$  over algebraically closed fields}\label{se:closed}

Throughout this section, the ground field $\bF$ will be assumed to be algebraically closed.

Due to Proposition \ref{L2}, in the sequel we will assume that the solvable radical of our non solvable and non simple irreducible quadratic Lie algebras satisfies ${\frr}\not= {\frr}^{\perp}$.

Recall that, up to isomorphism, there is a unique irreducible module $V(m)$ for the simple Lie algebra $\frsl_2(\bF)$ of dimension $m+1$ ($m\geq 0$). Moreover, the well-known Clebsch-Gordan formula  gives
\[
V(n)\otimes V(m)\cong V(n+m)\oplus V(n+m-2)\oplus\cdots\oplus V(\lvert n-m\rvert).
\]
For $n=m$, the summands on the right hand side belong alternatively to the second symmetric power and to the second exterior power of $V(n)$.
Besides, $V(n)$ is endowed with a nondegenerate symmetric bilinear form invariant under the action of $\frsl_2(\bF)$ if and only if $n$ is even, while if $n$ is odd, then $V(n)$ is endowed with a nondegenerate skew-symmetric and invariant bilinear form. These forms are unique up to scalar multiples.

\begin{proposition}\label{pr:01246}
Let $({\frg}= {\frs}\oplus {\frr},B)$ be a non solvable and non simple irreducible quadratic Lie algebra, where ${\frs}$ is a
Levi subalgebra of $\frg$ and ${\frr}$ is the solvable radical of $\frg$. Assume $\frr\ne\frr^\perp$.
If $\dim({\frg})\leq 13$, then:

\begin{enumerate}

\item ${\frs}$ is isomorphic to the simple Lie algebra $\frsl_2(\bF)$,
\item  if $\fru$ is an irreducible submodule of the ${\frs}(\cong\frsl_2(\bF))$-module
${{\frr}/{{\frr}^\perp}}$, then $\fru$ is, up to isomorphism, ${V}(n)$ where $n\in
\{0,1,2,4,6\}$.

\end{enumerate}
\end{proposition}

\begin{proof}
Since  $13 \geq \dim{\frg}\geq \dim \frs +\dim\frr^\perp= 2\dim{\frs}$, we get  $\dim{\frs}\leq 6$.  Therefore,  ${\frs}$ is isomorphic either to
$\frsl_2(\bF)\oplus \frsl_2(\bF)$ or  to  $\frsl_2(\bF)$.

Suppose first that  ${\frs}$ is isomorphic to $\frsl_2(\bF)\oplus \frsl_2(\bF)$. Then ${\frs}={\frs}_1\oplus{\frs}_2$ where ${\frs}_1$ and ${\frs}_2$ are ideals of ${\frs}$, both
isomorphic to $\frsl_2(\bF)$. Since $\frr^\perp$ is a $\frs$-submodule of $\frr$ isomorphic to $\frs$ (Proposition \ref{L1}), there exist three ${\frs}$-submodules
${\frm}_1, {\frm}_2, {\frm}_3$ of ${\frr}$ such that:

\begin{enumerate}
\item $\frm_i$ is isomorphic, as a module for $\frs$, to $\frs_i$, $i=1,2$,
\item $\dim({\frm}_3)\leq 1,$
\item ${\frr}= {\frm}_1\oplus {\frm}_2\oplus {\frm}_3,\,$  and ${\frr}^\perp= {\frm}_1\oplus
{\frm}_2$. Moreover, $[{\frs}_1,{\frm}_1]= {\frm}_1$, $[{\frs}_2,{\frm}_2]= {\frm}_2$ and
$[{\frs}_1,{\frm}_2]= [{\frs}_2,{\frm}_1]=[\frs,\frm_3]= 0$.
\end{enumerate}
Since $\frr^\perp$ is contained in $\frz(\frr)$ (Proposition \ref{L1}) and $\frr=\frr^\perp\oplus\frm_3$, it follows that $\frr$ is abelian, and hence $\frm_3$ is an ideal of $\frg$. Moreover, $\frm_3$ is nondegenerate, because $B(\frs,\frm_3)=B([\frs,\frs],\frm_3)=B(\frs,[\frs,\frm_3])=0$. The irreducibility of $\frg$ forces $\frm_3=0$.

Now, $\fri_1:=\frs_1\oplus\frm_1$ and $\fri_2:=\frs_2\oplus\frm_2$
are two non-zero ideals of $\frg$.
Moreover, by using the invariance of $B,$ it is clear that these two ideals
are nondegenerate, and this contradicts the fact that $\frg$ is irreducible.

We conclude that ${\frs}$ cannot be isomorphic to $\frsl_2(\bF)\oplus \frsl_2(\bF)$, and hence ${\frs}$ is
isomorphic to $\frsl_2(\bF)$.

Let $\fru$ be  an irreducible submodule of the ${\frs}$-module
${{\frr}/ {{\frr}^\perp}}$. Then, identifying $\frs\cong\frsl_2(\bF)$, $\fru$ is isomorphic to ${V}(n)$, with
$1\leq n\leq 6$.  Now if $\fru$ is isomorphic to ${V}(3)$
or ${V}(5)$, then there is no invariant symmetric nondegenerate bilinear form on $\fru$, so there must exist another irreducible submodule $\fru'$ of the ${\frs}$-module
${{\frr}/ {{\frr}^\perp}}$, isomorphic to $\fru$ and coupled with it by means of $B$. But then $\dim \frr/\frr^\perp = 8$ or $12$, and $\dim\frg > 13$.
We conclude that $\fru$ is isomorphic to ${V}(n)$, with $n\in
\{0,1,2,4,6\}$.
\end{proof}

\begin{example}\label{ex:Dm}
Let $m \in {\bN}$ be even. Then  ${V}(m)$ admits  a unique (up to scalars)
nondegenerate, symmetric and $\frsl_2(\bF)$-invariant bilinear form $T$. Denote by
$\pi$ the representation of $\frsl_2(\bF)$ in ${V}(m)$ and by ${\frd}(m)$ the
double extension of the abelian quadratic Lie algebra $({V}(m),T)$ by $\frsl_2(\bF)$ by means of $\pi$.

The bracket in ${\frd}(m)$ is given by
\[
[x+u+f,y+v+h]= [x,y]_{\frsl_2(\bF)} + \psi(u,v) - h\circ ad _{\frsl_2(\bF)}(x) + f\circ
ad _{\frsl_2(\bF)}(y),
\]
for any $x,y\in \frsl_2(\bF)$, $u,v \in {V}(m)$, and  $f,h \in \frsl_2(\bF)^*$,
where $\psi(u,v)\in \frsl_2(\bF)^*$ is defined by $\psi(u,v)(x):= T(\pi(x)(u),v)$ for any $x\in \frsl_2(\bF)$.

An invariant scalar product $B$ on ${\frd}(m)$ is defined by
\[
B(x+u+f,y+v+h):= T(x,y) + h(y) + f(x),
\]
for any $x,y\in \frsl_2(\bF)$, $u,v \in {V}(m)$, and $f,h \in \frsl_2(\bF)^*$. \qed
\end{example}

\begin{proposition}\label{pr:46}
Let $({\frg}= \frs\oplus {\frr},B)$ be a non solvable and non simple irreducible quadratic Lie algebra of dimension $\leq 13$, where $\frs$ is isomorphic to $\frsl_2(\bF)$ and $\frr$ is the solvable radical of $\frg$.

If the $\frs(\cong\frsl_2(\bF))$-module ${{\frr}/ {{\frr}^\perp}}$ contains a submodule isomorphic to
${V}(m)$, where
$m \in \{4,6\}$, then $\frg$ is isomorphic to ${\frd}(m)$.
\end{proposition}

\begin{proof}
If  the $\frs(\cong\frsl_2(\bF))$-module ${{\frr}/ {{\frr}^\perp}}$ contains a submodule isomorphic to ${V}(6)$,
then, by dimension count, there exists  an $\frs(\cong\frsl_2(\bF))$-submodule $\fru$ of ${\frr}$ isomorphic to ${V}(6)$
such that ${\frr}= {\frr}^\perp\oplus \fru$.
In this case, Proposition \ref{pr:R_Rperp} shows that $\frr^\perp=[{\frr},{\frr}]$.
It follows that ${{\frr}/ {{\frr}^\perp}}$ is an abelian Lie algebra and that $[\fru,\fru]= {\frr}^\perp$.
But ${V}(6)\wedge {V}(6)$ is isomorphic to ${V}(10)\oplus {V}(6)\oplus {V}(2)$, so we have
$\dim\Hom_{\frsl_2(\bF)}({V}(6)\wedge {V}(6),{V}(2))= 1$,  with a basis consisting of the projection $\psi$ onto $V(2)$  relative to the decomposition ${V}(6)\wedge {V}(6)\cong {V}(10)\oplus {V}(6)\oplus {V}(2)$).

Therefore, the Lie bracket in $\frg$ is completely determined by the action of $\frs$ on $\frg$ and by the unique (up to scalars) nonzero $\frs$-invariant bilinear map $\fru\times\fru\rightarrow \frr^\perp$. Hence, up to isomorphism, $\frg$ is uniquely determined and thus $\frg$ is isomorphic to $\frd(6)$.

If  the $\frs(\cong\frsl_2(\bF))$-module ${{\frr}/ {{\frr}^\perp}}$ contains a submodule isomorphic to ${V}(4)$ then, by dimension count, there are $\frs(\cong\frsl_2(\bF))$-modules $\fru$ and $\frv$ of $\frr$ such that $\frr=\frr^\perp\oplus \fru\oplus \frv$, with $\fru$ isomorphic to $V(4)$, and $\dim\frv\leq 2$. If $\dim \frv=2$, then by the invariance of $B$, the restriction of $B$ to $\frv$ is a symmetric nondegenerate bilinear form, and hence $\frv$ cannot be isomorphic to $V(1)$. It follows that $\frv$ is always a trivial $\frs$-module. Since $V(4)\wedge V(4)$ is isomorphic to $V(6)\oplus V(2)$, we get $[\fru,\fru]\subseteq \frr^\perp$. Also, by $\frs$-invariance $[\frv,\fru]\subseteq \fru$, so $\frs\oplus \fru\oplus\frr^\perp$ is a nondegenerate ideal of $\frg$, and hence it is the whole $\frg$ by irreducibility.

We conclude that $\frv= 0$,  and the same argument used in the case  $m= 6$ shows that $\frg$ is isomorphic to ${\frd}(4)$
\end{proof}

Therefore, because of Propositions \ref{pr:01246} and \ref{pr:46}, we are left with the situation in which our quadratic Lie algebra $\frg$ is, as a module for its Levi subalgebra $\frs\cong\frsl_2(\bF)$, a direct sum of copies of the irreducible modules $V(2)$, $V(1)$ and $V(0)$. Identify $\frs$ with the Lie algebra $\frsl(V)$ of endomorphisms of a two-dimensional vector space $V$. Then $V(2)$ is the adjoint module, and $V(1)$ is the natural module $V$. Gathering together the copies of irreducible modules of the same dimension, we may write, as in \cite{EO}:
\begin{equation}\label{eq:gslVVd}
\frg=\bigl(\frsl(V)\otimes \calJ\bigr)\oplus\bigl(V\otimes \calT\bigr)\oplus\frd,
\end{equation}
where $\frd$ is the sum of the trivial modules (this is the centralizer in $\frg$ of the subalgebra $\frsl(V)=\frsl(V)\otimes 1$, and hence it is a subalgebra), and $\calJ$ and $\calT$ are vector spaces, with $\calJ$ containing a distinguished element $1$, such that $\frsl(V)\otimes 1$ is the Levi subalgebra $\frs\cong\frsl(V)$ we have started with. (Here $\otimes$ means $\otimes_{\bF}$.)

The invariance of the form $B$ forces  the three summands $\frsl(V)\otimes \calJ$, $V\otimes \calT$ and $\frd$ to be orthogonal relative to $B$.
Besides, since the symmetric map $\frsl(V)\otimes \frsl(V)\rightarrow \bF$, $f\otimes g\mapsto \trace(fg)$, spans the space of $\frsl(V)$-invariant maps $\Hom_{\frsl(V)}(\frsl(V)\otimes\frsl(V),\bF)$ and the space $\Hom_{\frsl(V)}(V\otimes V,\bF)$ of $\frsl(V)$-invariant maps from $V\otimes V$ into the trivial module $\bF$ is one-dimensional and spanned by a fixed nonzero skew symmetric bilinear map $\bigl(u\vert v\bigr)$, we conclude that the restriction of $B$ to $\frsl(V)\otimes \calJ$ and to $V\otimes \calT$ present the form:
\begin{equation}\label{eq:B_BJ_BT}
\begin{split}
B(f\otimes a,g\otimes b)&=\trace(fg)B_\calJ(a,b),\\
B(u\otimes x,v\otimes y)&=(u\vert v)B_\calT(x,y),
\end{split}
\end{equation}
for any $f,g\in \frsl(V)$, $u,v\in V$, $a,b\in \calJ$ and $x,y\in \calT$, where $B_\calJ$ is a symmetric nondegenerate associative (i.e., $B_\calJ(a  b,c)=B_\calJ(a,b  c)$ for any $a,b,c\in \calJ$) symmetric bilinear form on $\calJ$, and $B_\calT$ is a nondegenerate skew-symmetric bilinear form on $\calT$. In particular $\dim \calT$ must be even.

Moreover, $\calJ=\bF 1\oplus \calJ_0$ and there is a distinguished element $c\in \calJ_0$ such that
\begin{equation}\label{eq:R_J}
\frr=\bigl(\frsl(V)\otimes \calJ_0\bigr)\oplus \bigl(V\otimes \calT\bigr)\otimes\frd,\quad \frr^\perp=\frsl(V)\otimes c.
\end{equation}

With the arguments in \cite{EO}, we see that the $\frsl(V)$-invariance of the Lie bracket in our Lie algebra $\frg$ in \eqref{eq:gslVVd} gives, for any $f,g\in\frsl(V)$, $u,v\in V$ and $d\in \frd$, the following conditions:
\begin{equation}\label{eq:g[]}
\begin{split}
[f\otimes a,g\otimes b]&=[f,g]\otimes a  b+2\trace(fg)D_{a,b},\\
[f\otimes a,u\otimes x]&=f(u)\otimes a\bullet x,\\
[u\otimes x,v\otimes y]&=\gamma_{u,v}\otimes \langle x\vert y\rangle +\bigl(u\vert v\bigr)d_{x,y},\\
[d,f\otimes a]&=f\otimes d(a),\\
[d,u\otimes x]&=u\otimes d(x),
\end{split}
\end{equation}
for suitable bilinear maps
\begin{equation}\label{eq:maps}
\begin{split}
\calJ\times \calJ\rightarrow \calJ,&\quad (a,b)\mapsto a  b,\quad\text{(symmetric),}\\
\calJ\times \calJ\rightarrow \frd,&\quad (a,b)\mapsto D_{a,b},\quad\text{(skew symmetric),}\\
\calJ\times \calT\rightarrow \calT,&\quad (a,x)\mapsto a\bullet x,\\
\calT\times \calT\rightarrow \calJ,&\quad (x,y)\mapsto \langle x\vert y\rangle\quad\text{(skew symmetric),}\\
\calT\times \calT\rightarrow \frd,&\quad (x,y)\mapsto d_{x,y}\quad\text{(symmetric),}\\
\frd\times \calJ\rightarrow \calJ,&\quad (d,a)\mapsto d(a),\\
\frd\times \calT\rightarrow \calT,&\quad (d,x)\mapsto d(x),
\end{split}
\end{equation}
such that $1  a=a$, $D_{1,a}=0$ and $1\bullet x=x$ for any $a\in \calJ$ and $x\in \calT$.

The Jacobi identity on $\frg$ forces that all these maps are invariant under the action of the Lie subalgebra $\frd$.

\begin{proposition}[{\cite[Theorem 2.2]{EO}}]\label{pr:conditions_g}
Under the hypotheses above, the following conditions hold:

\begin{enumerate}

\item $\calJ$ is a unital Jordan algebra with the product $a  b$, with unity $1$, and $D_{1,\calJ}=0$.

\item $\calT$ is a special unital Jordan module for $\calJ$, i.e., $(a  b)\bullet x=\frac{1}{2}\Bigl(a\bullet(b\bullet x)+b\bullet(a\bullet x)\Bigr)$ for any $a,b\in \calJ$ and $x\in \calT$.

\item $D_{a  b,c}+D_{b.c,a}+D_{c  a,b}=0$ and
$D_{a,b}(c)=a (b  c)-b (a  c )$ for any $a,b,c\in \calJ$.

\item $4D_{a,b}(x)=a\bullet (b\bullet x)-b\bullet(a\bullet x)$ for any $a,b\in \calJ$ and $x\in \calT$.

\item $4D_{a,\langle x\vert y\rangle}=-d_{a\bullet x,y}+d_{x,a\bullet y}$ for any $a\in \calJ$ and $x,y\in \calT$.

\item $2a \langle x\vert y\rangle=\langle a\bullet x\vert y\rangle +\langle x\vert a\bullet y\rangle$ for any $a\in \calJ$ and $x,y\in \calT$.

\item $d_{x,y}(a)=\langle a\bullet x\vert y\rangle -\langle x\vert a\bullet y\rangle$ for any $a\in \calJ$ and $x,y\in \calT$.

\item $d_{x,y}(z)-d_{z,y}(x)=\langle x\vert y\rangle\bullet z-\langle z\vert y\rangle \bullet x+2\langle x\vert z\rangle \bullet y$ for any $x,y,z\in \calT$.

\end{enumerate}
\end{proposition}

\begin{remark}
Conversely, the algebra defined on the vector space in \eqref{eq:gslVVd} by means of \eqref{eq:g[]} is a Lie algebra, provided that the bilinear maps in \eqref{eq:maps} are invariant under the action of the subalgebra $\frd$ and the conditions in Proposition \ref{pr:conditions_g} hold.

These Lie algebras are closely related to the Lie algebras graded by the nonreduced root system $BC_1$. The particular case in which $\calT=0$, i.e., $\frg$ is a direct sum of copies of the adjoint and the trivial module for $\frs$, goes back to \cite{Tits62}. (See \cite{EO} and the references therein.) \qed
\end{remark}

\begin{lemma}\label{le:J_T}
Let $\frg$ be an irreducible quadratic Lie algebra of dimension $\leq 13$ with Levi subalgebra $\frs\cong\frsl_2(\bF)$ and such that, as a module for $\frs$, $\frg$ is the sum of copies of the adjoint module, the two-dimensional irreducible module and the trivial one-dimensional module. Write $\frg$ and its solvable radical $\frr$ as in \eqref{eq:gslVVd} and \eqref{eq:R_J}. Then the following conditions hold:
\begin{itemize}
\item $\calJ$ is a unital commutative and associative algebra and $\calJ_0$ is its nilpotent radical.
\item $D_{a,b}(\calJ)=0$ for any $a,b\in \calJ$.
\item $D_{1,\calJ}=0$, $c  \calJ_0=0$, $c\bullet \calT=0$ and $d(c)=0$ for any $d\in \frd$.
\item If $\calT=0=\frd$, then $c\in\calJ_0^2$.
\item $\dim\calT$ is even.
\end{itemize}
\end{lemma}

\begin{proof}
Since $\frr$ in \eqref{eq:R_J} is an ideal of $\frg$, it follows that $\calJ_0$ is an ideal of $\calJ$. Moreover, there is no idempotent $0\ne e=e^2$ in $\calJ_0$, as this would imply that $\frsl(V)\otimes e$ is a subalgebra of $\frg$ contained in the solvable radical $\frr$ and isomorphic to $\frsl_2(\bF)$. Therefore, $\calJ_0$ is a nilpotent ideal of the Jordan algebra $\calJ$. (For the basic facts on Jordan algebras the reader may consult \cite{JacobsonJordan} or \cite{Schafer}.)

Also, since $\dim \frg\leq 13$, we have $2\leq\dim \calJ\leq 4$ and hence $1\leq \dim \calJ_0\leq 3$, so that $\calJ_0^4=0$. Hence either $\calJ_0^3=0$, so in particular $\calJ_0$ is associative, or $\calJ_0^2$ has codimension $1$ in $\calJ_0$. In the latter case, $\calJ_0=\bF a+\calJ_0^2$, and this shows $\calJ_0^2=\bF a^2+ \calJ_0^3$, and $\calJ_0^3=\bF a^3$, so that $\calJ_0$ is generated by $a$ with $a^4=0$. In any case, $\calJ_0$ is a commutative and associative algebra, and hence $\calJ$ is a unital commutative and associative algebra. In particular, this gives $D_{a,b}(\calJ)=0$ for any $a,b\in \calJ$, by Proposition \ref{pr:conditions_g}\,(3).

Besides, since $\frr^\perp\subseteq \frz(\frr)$ (Proposition \ref{L1}), we have
$c  \calJ_0=0$, $D_{c,\calJ}=0$, $c\bullet \calT=0$ and $d(c)=0$ for any $d\in \frd$ and since we assume $\frr^\perp\subseteq [\frr,\frr]$ (because of Proposition \ref{pr:R_Rperp}), we have $c\in \calJ_0^2$ if $\calT=0=\frd$.

As there is no nonzero $\frsl_2(\bF)$-invariant symmetric bilinear form on $V(1)$, the number of copies of $V$ (i.e., $\dim\calT$) is even.
\end{proof}

\begin{proposition}\label{pr:dimJ_4}
Let $({\frg}= \frs\oplus {\frr},B)$ be a non solvable and non simple irreducible quadratic Lie algebra of dimension $\leq 13$, where $\frs$ is isomorphic to $\frsl_2(\bF)$ and $\frr$ is the solvable radical of $\frg$.

If, as a module for $\frs(\cong\frsl_2(\bF))$, $\frg$ contains four copies of the adjoint module, then $\frg$ is isomorphic to $\frsl_2(\bF)\otimes \calA$, where $\calA$ is a unital commutative associative algebra isomorphic to either $\bF[X]/(X^4)$, or to $\bF[X,Y]/(X^2,Y^2)$, with the usual bracket $[x\otimes a,y\otimes b]=[x,y]\otimes ab$ for $x,y\in \frsl_2(\bF)$ and  $a,b\in\calA$.
\end{proposition}

\begin{proof}
Under our hypotheses we have $\dim \calJ=4$, $\calT=0$ and $\dim\frd\leq 1$ in \eqref{eq:gslVVd}. If there are elements $a,b\in \calJ_0$ such that $D_{a,b}\ne 0$. Then $\frd=\bF D_{a,b}$, and hence $\frd$ acts trivially on $\calJ$ by Lemma \ref{le:J_T}, so that $\frd$ is the center of $\frg$, and thus $\frd$ is a nondegenerate ideal of $\frg$, a contradiction. Therefore $D_{\calJ,\calJ}=0$, but then $\frsl(V)\otimes \calJ$ is a nondegenerate ideal of $\frg$, so it is the whole $\frg$ by irreducibility.

As above, if the codimension of $\calJ_0^2$ in $\calJ_0$ is $1$, we get that $\calJ$ is generated by an element $a$ with $a^4=0$, so $\calJ$ is isomorphic to $\bF[X]/(X^4)$. Otherwise $\calJ_0^2=\bF c$, and the bilinear map $\bigl(\calJ_0/\calJ_0^2\bigr)\times \bigl(\calJ_0/\calJ_0^2\bigr)\rightarrow \calJ_0^2$ given by $(a+\calJ_0^2,b+\calJ_0^2)\mapsto ab$ is bilinear, symmetric and nonzero. Hence, since $\bF$ is algebraically closed, there are elements $a,b\in \calJ_0\setminus \calJ_0^2$ such that either $a^2=b^2=0$ and $a  b\ne 0$ (if the rank of this bilinear map is $2$) and in this case $\calJ$ is isomorphic to $\bF[X,Y]/(X^2,Y^2)$, or $a^2\ne 0$ and $a  b=b^2=0$ (if the rank is $1$). In the latter case, the bilinear form $B_\calJ$ in \eqref{eq:B_BJ_BT} satisfies $B_\calJ(b,\calJ_0)=B_\calJ(1,b  \calJ_0)=0$, and also $B_\calJ(a,a^2)=B_\calJ(1,a^3)=0$ and $B_\calJ(a^2,a^2)=B_\calJ(1,a^4)=0$. Thus $B_\calJ(a^2,\calJ_0)=0$ and we get a contradiction with $B_\calJ$ being nondegenerate, so the latter case cannot occur.
\end{proof}

\begin{remark}
The two possibilities in Proposition \ref{pr:dimJ_4} are indeed quadratic Lie algebras, with invariant scalar product given by $B(f\otimes a,g\otimes b)=\trace(fg)B_{\calA}(a,b)$ for any $f,g\in \frsl(V)$ and $a,b\in \calA$, where $B_{\calA}(x^i,x^j)=1$ if and only if $i+j=3$ in the first case, with $x$ being the class of $X$ modulo $(X^4)$, and where $B_{\calA}(1,xy)=1=B_{\calA}(x,y)=1$ and $B_{\calA}$ gives $0$ for any other basic elements $1,x,y,xy$, where $x$ and $y$ are the classes of $X$ and $Y$ modulo $(X^2,Y^2)$ in the second case. \qed
\end{remark}

\smallskip

\begin{example}\label{ex:V2V1V1}
Let $V$ be a two-dimensional vector space as above, endowed with a nondegenerate skew-symmetric bilinear form $(.\vert .)$. On the vector space $\fra=\frsl(V)\oplus (V\otimes V)$ consider the symmetric nondegenerate bilinear form $B_\fra$ such that $B_\fra(\frsl(V),V\otimes V)=0$, $B_\fra(f,g)=\trace(fg)$ and $B_\fra(u_1\otimes v_1,u_2\otimes v_2)=(u_1\vert u_2)(v_1\vert v_2)$, for $f,g\in\frsl(V)$ and $u_1,u_2,v_1,v_2\in V$. Consider the vector space $\fra$ as an \emph{abelian} Lie algebra.

The linear map $\varphi:\frsl(V)\rightarrow \Der(\fra,B_\fra)$ given by $\varphi(f)(g)=[f,g]$, $\varphi(f)(u\otimes v)=f(u)\otimes v$, for any $f,g\in \frsl(V)$ and $u,v\in V$, is a Lie algebra homomorphism. The double extension $T_\varphi(\fra,B_\fra,\frsl(V))$ is an irreducible quadratic Lie algebra, with a Levi subalgebra isomorphic to $\frsl_2(\bF)$ and such that, as a module for this subalgebra, it decomposes as the direct sum of three copies of the adjoint module, and two copies of the natural two-dimensional module. Its dimension is then $13$. \qed
\end{example}

\begin{proposition}\label{pr:dimJ_3}
Let $({\frg}= \frs\oplus {\frr},B)$ be a non solvable and non simple irreducible quadratic Lie algebra of dimension $\leq 13$, where $\frs$ is isomorphic to $\frsl_2(\bF)$ and $\frr$ is the solvable radical of $\frg$.

If, as a module for $\frs(\cong\frsl_2(\bF))$, $\frg$ contains three copies of the adjoint module, then either $\frg$ is isomorphic to $\frsl_2(\bF)\otimes \calA$,
 where $\calA$ is a unital commutative associative algebra isomorphic to  $\bF[X]/(X^3)$, or $\frg$ is isomorphic to the Lie algebra
in Example \ref{ex:V2V1V1}.
\end{proposition}

\begin{proof}
We are in the situation where $\dim \calJ=3$ in \eqref{eq:gslVVd}. Also, $\calJ_0^2\ne 0$. Otherwise $B_\calJ(\calJ_0,\calJ_0)=B_\calJ(1,\calJ_0^2)$ would be $0$ and $B_\calJ$ would be degenerate. The arguments in the proof of Proposition \ref{pr:dimJ_4} show that $\calJ$
is isomorphic to $F[X]/(X^3)$. Take $a\in \calJ_0\setminus \calJ_0^2$, then $a^3=0$ and $\frr^\perp=\frsl(V)\otimes a^2$ (so we may take $c=a^2$).
Also, since $D_{1,\calJ}=D_{a^2,\calJ}=0$ and $D_{.,.}$ is skew-symmetric, we get $D_{\calJ,\calJ}=0$. If $\calT=0$, then it follows that $\frsl(V)\otimes \calJ$
is a nondegenerate ideal of $\frg$, and we obtain that $\frg$ is isomorphic to $\frsl_2(\bF)\otimes\calA$, with $\calA=\bF[X]/(X^3)$.

Otherwise, by dimension count, we must have $\dim \calT=2$ and $\dim\frd=0$. Since $\frr^\perp\subseteq \frz(\frr)$ (Proposition \ref{L1}),
we have $0=a^2\bullet \calT=a\bullet (a\bullet \calT)$. The map $\calT\times \calT\rightarrow \calJ$, $(x,y)\mapsto \langle x\vert y\rangle$
is skew-symmetric. If it were trivial, $V\otimes \calT$ would be a nondegenerate ideal of $\frg$, a contradiction.
Hence $\langle \calT\vert \calT\rangle\ne 0$. If $a\bullet \calT$ is nonzero, then there is an element $x\in \calT$ such
that $\{x,a\bullet x\}$ is a basis of $\calT$. Then we must have $\langle x\vert a\bullet x\rangle\ne 0$ but,
because of Proposition \ref{pr:conditions_g}, we have
$0=\langle a\bullet x\vert x\rangle -\langle x\vert a\bullet x\rangle=-2\langle x\vert a\bullet x\rangle =d_{x,x}(a)$,
which is trivial since $\frd=0$ and thus $d_{x,x}=0$.
Thus $a\bullet \calT=0$ and hence $[\frsl(V)\otimes \calJ_0, V\otimes \calT]=0$. Again by Proposition \ref{pr:conditions_g} we have $\langle \calT\vert \calT\rangle   \calJ_0=0$, so $\langle \calT\vert \calT\rangle =\bF a^2$. Therefore we may pick up a basis $\{x,y\}$ of $\calT$ with $\langle x\vert y\rangle =a^2$, and the multiplication of $\frg$ is completely determined. We conclude that there is a unique possibility, up to isomorphism, and the result follows.
\end{proof}

We are then left with the case $\dim \calJ=2$. Then $\calJ$ is isomorphic to the algebra of dual numbers $\bF[X]/(X^2)$, i.e., $\calJ=\bF 1+\bF c$ with $c^2=0$. By dimension count, either $\calT=0$ or $\dim \calT=2$, and if $\calT=0$, $\frsl(V)\otimes \calJ$ is a nondegenerate ideal, and we get the trivial $T^\star$ extension in Proposition \ref{pr:R_Rperp}. Therefore we may assume $\dim \calT=2$. Again by dimension count $\dim\frd\leq 3$. Let us start with several examples.

\begin{example}\label{ex:V1V1}
Let $V$ be a two-dimensional vector space as above, endowed with a nondegenerate skew-symmetric bilinear form $(.\vert .)$. On the vector space $\frb=V\otimes V$ consider the symmetric nondegenerate bilinear form $B_\frb$ such that  $B_\frb(u_1\otimes u_2,v_1\otimes v_2)=(u_1\vert v_1)(u_2\vert v_2)$ for any $u_1,u_2,v_1,v_2\in V$. Consider the vector space $\frb$ as an \emph{abelian} Lie algebra. The linear map $\varphi:\frsl(V)\rightarrow \Der(\frb,B_\frb)$ given by  $\varphi(f)(u\otimes v)=f(u)\otimes v$, for any $f\in \frsl(V)$ and $u,v\in V$, is a Lie algebra homomorphism. The double extension $T_\varphi(\frb,B_\frb,\frsl(V))$ is a quadratic Lie algebra, with a Levi subalgebra isomorphic to $\frsl_2(\bF)$ and such that, as a module for this subalgebra, it decomposes as the direct sum of two copies of the adjoint module, and two copies of the natural two-dimensional module. Its dimension is then $10$. \qed
\end{example}

\begin{example}\label{ex:V1V1ddiagonal}
Let $V$ and $(.\vert.)$ as before. Consider the abelian Lie algebra $\frb=V\otimes V$ as in Example \ref{ex:V1V1}. Fix a basis $\{u,v\}$ of $V$ with $(u\vert v)=1$, and the one-dimensional Lie algebra $\bF d$. Let $\varphi:\bF d\rightarrow \Der(\frb)$ be the Lie algebra homomorphism such that $\varphi(d)(u_1\otimes u)=u_1\otimes u$, $\varphi(d)(u_1\otimes v)=-u_1\otimes v$, for any $u_1\in V$. Then the double extension $\hat\frb=T_\varphi(\frb,B_\frb,\bF d)=\bF d\ltimes_{\bar\varphi}(\frb\oplus \bF d^*)$ ($d^*\in (\bF d)^*$ with $d^*(d)=1$)
is a solvable Lie algebra of dimension $6$, endowed with an invariant scalar product $B_{\hat\frb}$ such that $B_{\hat\frb}(\bF d+ \bF d^*,\frb)=0$, $B_{\hat\frb}\vert_\frb =B_\frb$ and $B_{\hat\frb}(d,d^*)=1$, $B_{\hat\frb}(d,d)=0=B_{\hat\frb}(d^*,d^*)$.

Consider now the Lie algebra homomorphism $\phi:\frsl(V)\rightarrow \Der(\hat\frb,B_{\hat\frb})$ given by $\phi(s)(\bF d+\bF d^*)=0$, $\phi(s)(u_1\otimes u_2)=s(u_1)\otimes u_2$ for any $s\in \frsl(V)$ and $u_1,u_2\in V$. The double extension $T_{\phi}(\hat\frb,B_{\hat\frb},\frsl(V))$ is an irreducible quadratic Lie algebra, with a Levi subalgebra isomorphic to $\frsl_2(\bF)$ and such that, as a module for this subalgebra, it decomposes as the direct sum of two copies of the adjoint module, two copies of the natural two-dimensional module, and two copies of the trivial one-dimensional module. Moreover, the subalgebra formed by the two copies of the trivial one-dimensional module act diagonally on the sum of the two natural modules. \qed
\end{example}

\begin{example}\label{ex:V1V1dnilpotent}
Let $V$ and $(.\vert.)$ as before. Consider the abelian Lie algebra $\frb=V\otimes V$ as in Example \ref{ex:V1V1}. Fix, as before, a basis $\{u,v\}$ of $V$ with $(u\vert v)=1$, and the one-dimensional Lie algebra $\bF d$. Let $\varphi':\bF d\rightarrow \Der(\frb)$ be the Lie algebra homomorphism, such that $\varphi'(d)(u_1\otimes u)=u_1\otimes v$, $\varphi'(d)(u_1\otimes v)=0$, for any $u_1\in V$. Then the double extension $\hat\frb'=T_{\varphi'}(\frb,B_\frb,\bF d)=\bF d\ltimes_{\bar\varphi'}(\frb\oplus \bF d^*)$ ($d^*\in (\bF d)^*$ with $d^*(d)=1$)
is a solvable Lie algebra of dimension $6$, endowed with an invariant scalar product $B_{\hat\frb'}$ given by the same formulas as $B_{\hat\frb}$ in Example \ref{ex:V1V1ddiagonal}.

Consider now the Lie algebra homomorphism $\phi':\frsl(V)\rightarrow \Der(\hat\frb',B_{\hat\frb})$ given by the same formulas as in Example \ref{ex:V1V1ddiagonal}.  The double extension $T_{\phi'}(\hat\frb',B_{\hat\frb'},\frsl(V))$ is an irreducible quadratic Lie algebra, with a Levi subalgebra isomorphic to $\frsl_2(\bF)$ and such that, as a module for this subalgebra, it decomposes as the direct sum of two copies of the adjoint module, two copies of the natural two-dimensional module, and two copies of the trivial one-dimensional module. Moreover, the subalgebra formed by the two copies of the trivial one-dimensional module act in a nilpotent way on the sum of the two natural modules. \qed
\end{example}

\begin{example}\label{ex:V1V1d}
Again, let $V$, $(.\vert.)$ and $\{u,v\}$ as before. Consider the nilpotent Lie algebra $\frc=(V\otimes V)\oplus\bF d$ with multiplication given by $[u_1\otimes u,v_1\otimes u]=(u_1\vert v_1)d$, $[u_1\otimes u,v_1\otimes v]=0=[u_1\otimes v,v_1\otimes v]$, and $[d,u_1\otimes u]=u_1\otimes v$, $[d,u_1\otimes v]=0$, for any $u_1,v_1\in V$. The Lie algebra $\frc$ is endowed with an invariant scalar product $B_{\frc}$ such that $B_{\frc}(V\otimes V,d)=0$, $B_{\frc}(u_1\otimes u_2,v_1\otimes v_2)=(u_1\vert v_1)(u_2\otimes v_2)$ for any $u_1,v_1,u_2,v_2\in V$, and $B_{\frc}(d,d)=-1$.

Consider the Lie algebra homomorphism $\varphi:\frsl(V)\rightarrow \Der(\frc,B_\frc)$ given by $\varphi(f)(u_1\otimes u_2)=f(u_1)\otimes u_2)$, $\varphi(f)(d)=0$, for any $f\in \frsl(V)$ and $u_1,u_1\in V$. The double extension $T_\varphi(\frc,B_\frc,\frsl(V))$ is an irreducible quadratic Lie algebra, with a Levi subalgebra isomorphic to $\frsl_2(\bF)$ and such that, as a module for this subalgebra, it decomposes as the direct sum of two copies of the adjoint module, two copies of the natural two-dimensional module, and a trivial one-dimensional module, which acts in a nilpotent way on the sum of the two natural modules. \qed
\end{example}

\begin{proposition}\label{pr:dimJ_2}
Let $({\frg}= \frs\oplus {\frr},B)$ be a non solvable and non simple irreducible quadratic Lie algebra of dimension $\leq 13$, where $\frs$ is isomorphic to $\frsl_2(\bF)$ and $\frr$ is the solvable radical of $\frg$, with $\frr\ne \frr^\perp$.

If, as a module for $\frs(\cong\frsl_2(\bF))$, $\frg$ contains two copies of the adjoint module, then $\frg$ is isomorphic to the Lie algebra in one of the Examples \ref{ex:V1V1}, \ref{ex:V1V1ddiagonal}, \ref{ex:V1V1dnilpotent} or \ref{ex:V1V1d}.
\end{proposition}

\begin{proof}
We are in the situation where $\dim \calJ=2$ in \eqref{eq:gslVVd}. Here $\calJ$ is isomorphic to $\bF[X]/(X^2)$, i.e.,
$\calJ=\bF 1+\bF c$ with $\calJ_0=\bF c$ and $c^2=0$. Since $D_{1,\calJ}=0$, we have $D_{\calJ,\calJ}=0$.
If $\calT=0$, then $\frsl(V)\otimes\calJ$ is a nondegenerate ideal, so it is the whole $\frg$, a contradiction with $\frr\ne\frr^\perp$. Hence we have
 $\dim \calT=2$ and $\dim \frd\leq 3$.
Also, since $\frr^\perp=\frsl(V)\otimes c\subseteq \frz(\frr)$
we have $\calJ_0\bullet \calT=0$ and $d(\calJ)=0$ for any $d\in \frd$.
By irreducibility of $\frg$, $(V\otimes \calT)\oplus\frd$ cannot be an ideal of $\frg$,
and hence $\langle \calT\vert \calT\rangle$ is not zero. Thus, for any basis $\{x,y\}$ of $\calT$,
we have $\calJ_0=\bF \langle x\vert y\rangle$.

$\frd$ is a solvable quadratic Lie algebra with $\dim\frd\leq 3$, and hence it is  abelian.
Indeed,
by Proposition \ref{LEM},  ${\frz}({\frd})\not= \{0\}$ and hence $\dim[\frd,\frd]\leq\dim\wedge^2\bigl(\frd/\frz(\frd)\bigr)\leq 1$. Thus $\dim\frd=\dim[\frd,\frd]+\dim[\frd,\frd]^\perp=\dim[\frd,\frd]+\dim\frz(\frd)$. Thus the codimension of $\frz(\frd)$ is $\leq 1$ and $\frd$ is abelian.

The invariance of the scalar product forces that the action of $\frd$ on $\calT$ gives a Lie
algebra homomorphism $\frd\rightarrow \frsp(\calT,B_\calT)$ (the symplectic Lie algebra of the skew-symmetric bilinear form $B_\calT$ as in \eqref{eq:B_BJ_BT}). The image of this homomorphism is an abelian Lie algebra of $\frsp(\calT,B_\calT)\cong\frsl_2(\bF)$, so its dimension is at most $1$. On the other hand, the kernel of this homomorphism: $\{d\in\frd : [d,V\otimes\calT]=0\}$, coincides with the center of $\frg$, which is isotropic because of the irreducibility of $\frg$. Since $\dim \frd\leq 3$, we conclude that the kernel has dimension at most $1$, and hence we have $\dim\frd\leq 2$.

If $\frd=0$ the Lie algebra $\frg$ is uniquely determined, and hence it is isomorphic to the Lie algebra in Example \ref{ex:V1V1}. Otherwise, $\frd$ is not an ideal ($\frg$ is irreducible) so the image of the homomorphism above is not trivial, i.e., there is an element $d\in \frd$ such that $d(\calT)\ne 0$.

If $d$ acts faithfully on $\calT$, we may scale $d$ and take a basis $\{x,y\}$ of $\calT$ with $B_\calT(x,y)=1$, $d(x)=x$ and $d(y)=-1$ (note that the action of $d$ lies in $\frsp(\calT,B_\calT)$ and hence its trace is $0$). But the symmetric bilinear map $\calT\times \calT\rightarrow \frd$, $(z,t)\mapsto d_{z,t}$ is invariant under the action of $\frd$. Hence, $[d,d_{x,x}]=d_{d(x),x}+d_{x,d(x)}=2d_{x,x}$. Since $\frd$ is abelian, we get $d_{x,x}=0$, and also $d_{y,y}=0$. But Proposition \ref{pr:conditions_g} shows that $d_{x,y}(x)=d_{x,x}(y)=0$, $d_{x,y}(y)=d_{y,y}(x)=0$, so $d_{x,y}(\calT)=0$. If $d_{x,y}$ were $0$ we would have $d_{\calT,\calT}=0$ and get the nondegenerate ideal $(\frsl(V)\otimes \calJ)\oplus(V\otimes \calT)$ of $\frg$, a contradiction with the irreducibility of $\frg$. Hence $0\ne d_{x,y}$ spans the kernel of the Lie algebra homomorphism $\frd\rightarrow \frsp(\calT,B_\calT)$. We conclude that $\frd=\bF d+\bF d_{x,y}$ and the Lie bracket in $\frg$ is uniquely determined. Therefore $\frg$ must be, up to isomorphism, the Lie algebra in Example \ref{ex:V1V1ddiagonal}.

On the other hand, if the action of $d$ on $\calT$ is nilpotent, we may take a basis $\{x,y\}$ of $\calT$ with $d(x)=y$ and $d(y)=0$. Proposition \ref{pr:conditions_g} shows that $0=[d,d_{x,x}]=2d_{x,y}$ and $0=[d,d_{x,y}]=d_{y,y}$. Since $d_{\calT,\calT}$ must be nonzero as above, we get $d_{\calT,\calT}=\bF d_{x,x}\ne 0$. In case $d_{x,x}(\calT)=0$, the multiplication is completely determined and $\frg$ is isomorphic to the Lie algebra in Example \ref{ex:V1V1dnilpotent}.

Otherwise $d_{x,x}(\calT)\ne 0$, and since the image of $\frd$ in $\frsp(\calT,B_\calT)$ is one-dimensional, we may assume, scaling $y$ if necessary, that $d=d_{x,x}$. Then, choosing elements $u,v\in V$ with $(u\vert v)=1$ we have
\begin{multline*}
B(d_{x,x},d_{x,x})=B(d_{x,x},[u\otimes x,v\otimes x])=B([d_{x,x},u\otimes x],v\otimes x)\\
=B(u\otimes y,v\otimes x)=B_\calT(y,x)\ne 0,
\end{multline*}
and $[\frg,\frg]=(\frsp(V)\otimes \calJ)\oplus(V\otimes \calT)\oplus\bF d_{x,x}$ is a nondegenerate ideal of $\frg$. By irreducibility this is the whole $\frg$, so $\frd=\bF d_{x,x}$. Again the multiplication in $\frg$ is completely determined, and hence $\frg$ is isomorphic to the Lie algebra in Example \ref{ex:V1V1d}.
\end{proof}

\smallskip

The next result summarizes our classification:

\begin{theorem}
The complete list, up to isomorphisms, of the non solvable irreducible quadratic Lie algebras $\frg$ with $\dim\frg\leq 13$ is the following:
\begin{enumerate}
\item $\dim\frg=3$: the simple Lie algebra $\frsl_2(\bF)$,
\item $\dim\frg=6$: the trivial $T^*$-extension $T_0^\star(\frsl_2(\bF))$,
\item $\dim\frg=8$: the simple Lie algebra $\frsl_3(\bF)$,
\item $\dim\frg=9$: the `scalar extension' $\frsl_2(\bF)\otimes \bF[X]/(X^3)$,
\item $\dim\frg=10$: the simple Lie algebra $\frsp_4(\bF)\cong\frso_5(\bF)$ and the Lie algebra in Example \ref{ex:V1V1},
\item $\dim\frg=11$: the double extension $\frd(4)$ and the Lie algebra in Example \ref{ex:V1V1d},
\item $\dim\frg=12$: the `scalar extensions' $\frsl_2(\bF)\otimes \bF[X]/(X^4)$ and $\frsl_2(\bF)\otimes \bF[X,Y]/(X^2,Y^2)$ and the Lie algebras in Examples \ref{ex:V1V1ddiagonal} and \ref{ex:V1V1dnilpotent},
\item $\dim\frg=13$: the double extension $\frd(6)$ and the Lie algebra in Example \ref{ex:V2V1V1}.
\end{enumerate}
\end{theorem}


\section{Classification of the non solvable irreducible   quadratic Lie algebras of dimension $\leq 13$  over the real field}\label{se:real}

The situation over the real field is more involved than for algebraically closed fields. However, many of the arguments in the previous section are valid. The ground field in this section is always the real field $\bR$.

Let us recall first that for the split simple real Lie algebra $\frsl_2(\bR)$ there exists an irreducible module $V_\bR(n)$ of dimension $n+1$ for each $n\in\bZ_{\geq 0}$, and these exhaust the list of finite dimensional irreducible modules, up to isomorphism. However, for the compact Lie algebra $\frsu_2(\bR)$ of skew-hermitian $2\times 2$ complex matrices, the irreducible modules are, up to isomorphism, the following (see, for instance, \cite{BenkartOsborn}):
\begin{itemize}
\item The modules $V(n)$ for $\frsl_2(\bC)$ for $n$ odd, which remain irreducible as modules for $\frsu_2(\bR)\subseteq \frsl_2(\bC)$. (Note that $\dim_\bR V(n)=2(n+1)$.)

    In particular, $V(1)\simeq \bC^2$ is the natural module for $\frsu_2(\bR)$, of (real) dimension $4$.

\item For each even $n\geq 0$ a module $W(n)$ of dimension $n+1$, such that $W(n)\otimes_\bR \bC$ is isomorphic to the module $V(n)$ for $\frsl_2(\bC)\simeq \frsu_2(\bR)\otimes_\bR\bC$. In particular, $W(0)$ is the trivial one-dimensional module for $\frsu_2(\bR)$, and $W(2)$ is the adjoint module.
\end{itemize}

The natural module $W=V(1)\simeq \bC^2$ is endowed with the natural hermitian form $h:W\times W\rightarrow \bC$ with
\[
h\left(\left(\begin{smallmatrix} x_1\\ x_2\end{smallmatrix}\right),
    \left(\begin{smallmatrix} y_1\\ y_2\end{smallmatrix}\right)\right)
    =x_1\overline{y_1}+x_2\overline{y_2},
\]
and hence we will identify $\frsu_2(\bR)\simeq \frsu(W,h)$. The unique, up to scalars, symmetric nondegenerate $\frsu_2(\bR)$-invariant bilinear form on $W$ is given by
\begin{equation}\label{eq:B_W}
B_W(u,v)=2\Re\bigl(h(u,v)\bigr)=h(u,v)+h(v,u),
\end{equation}
for $u,v\in W$. And the unique (up to scalars) $\frsu_2(\bR)$-invariant bilinear map $W\times W\rightarrow \frsu_2(\bR)\simeq \frsu(W,h)$ is given by
\begin{equation}\label{eq:Gamma}
\Gamma(u,v)=h(.,v)u-h(.,u)v,
\end{equation}
for $u,v\in W$. Note that for $s\in \frsu(V,h)$ and $u,v\in W$,
\begin{equation}\label{eq:Gamma_B_W}
\begin{split}
\trace\bigl(s\Gamma(u,v)\bigr)&=\trace\left(s\bigl(h(.,u)v-h(.,v)u\bigr)\right)\\
    &=\trace\left(h(.,v)s(u)-h(.,u)s(v)\right)\\
    &=h\bigl(s(u),v\bigr)-h\bigl(s(v),u\bigr)\\
    &=h\bigl(s(u),v\bigr)+h\bigl(v,s(u)\bigr)\\
    &=2\Re\bigl(h(s(u),v)\bigr)=B_W\bigl(s(u),v\bigr).
\end{split}
\end{equation}

The same arguments used in the proof of Proposition \ref{pr:01246} give the following result:

\begin{proposition}\label{pr:real01246}
Let $(\frg=\frs\oplus\frr,B)$ be a non solvable and non simple irreducible quadratic Lie algebra over $\bR$, where $\frs$ is a Levi subalgebra and $\frr$ is the solvable radical of $\frg$. Assume $\frr\ne \frr^\perp$. If $\dim\frg\leq 13$, then:
\begin{enumerate}

\item ${\frs}$ is either isomorphic to $\frsl_2(\bR)$ or to $\frsu_2(\bR)$;

\item  if $\frs$ is isomorphic to $\frsl_2(\bR)$ and $\fru$ is an irreducible submodule of the ${\frs}(\cong\frsl_2(\bR))$-module
${{\frr}/{{\frr}^\perp}}$, then $\fru$ is, up to isomorphism, ${V_\bR}(n)$ where $n\in
\{0,1,2,4,6\}$;

\item if $\frs$ is isomorphic to $\frsu_2(\bR)$ and $\fru$ is an irreducible submodule of the $\frs(\cong\frsu_2(\bR))$-module $\frr/\frr^\perp$, then, up to isomorphism, $\fru\in\{W(0),W(2),W(4),W(6),V(1)\}$.

\end{enumerate}
\end{proposition}

In case (2) above the situation is close to what happens over algebraically closed fields. Denote by $\frd_\bR(4)$ and $\frd_6(\bR)$ the double extensions as in Example \ref{ex:Dm} but with $\bF=\bR$. On the other hand, the same construction applies for $\frsu_2(\bR)$ and the modules $W(m)$ for even $m$. The corresponding double extensions will be denoted by $\check\frd_\bR(m)$.

Then the proof of Proposition \ref{pr:46} can be easily modified to yield the following result.

\begin{proposition}\label{pr:real46}
Let $(\frg=\frs\oplus\frr,B)$ be a non solvable and non simple irreducible quadratic Lie algebra over $\bR$ of dimension $\leq 13$, where $\frs$ is a Levi subalgebra and $\frr$ is the solvable radical of $\frg$.
\begin{enumerate}

\item If ${\frs}$ is isomorphic to $\frsl_2(\bR)$ and the $\frs$-module $\frr/\frr^\perp$ contains a submodule isomorphic to $V_\bR(m)$, $m=4$ or $6$, then $\frg$ is isomorphic to $\frd(m)$.

\item If $\frs$ is isomorphic to $\frsu_2(\bR)$ and the $\frs$-module $\frr/\frr^\perp$ contains a submodule isomorphic to $W(m)$, $m=4$ or $6$, then $\frg$ is isomorphic to $\check\frd_\bR(m)$.

\end{enumerate}
\end{proposition}

Therefore, as in Section \ref{se:closed}, assuming $\dim\frg\leq 13$, we are left with the situation in which our quadratic Lie algebra $\frg$ is, as a module for its Levi subalgebra $\frs$ (isomorphic to either $\frsl_2(\bR)$ or $\frsu_2(\bR)$),  a direct sum of copies of the adjoint module, the natural module  and the trivial one-dimensional module. For $\frs\cong\frsl_2(\bR)$, we may identify $\frs$ with $\frsl(V)$, with $V$ the natural module, and equation \eqref{eq:gslVVd} remains valid here. On the other hand, for $\frs\cong\frsu_2(\bR)$, if $W$ is the natural module for $\frsu_2(\bR)\cong\frsu(W,h)$, by dimension count we get one of the following analogues of \eqref{eq:gslVVd}
\begin{equation}\label{eq:gsuWd}
\frg=\bigl(\frsu(W,h)\otimes\calJ\bigr)\oplus\frd,
\end{equation}
or
\begin{equation}\label{eq:gsuWWd}
\frg=\bigl(\frsu(W,h)\otimes\calJ\bigr)\oplus W\oplus \frd.
\end{equation}
(Since the dimension of $W$ is $4$, at most one copy of $W$ may appear in any decomposition of $\frg$ into a direct sum of irreducible module for the Levi subalgebra $\frs\cong\frsu(W,h)$.)

\begin{proposition}\label{pr:real_dimJ_4}
Let $({\frg}= \frs\oplus {\frr},B)$ be a non solvable and non simple irreducible quadratic Lie algebra over $\bR$ of dimension $\leq 13$, where $\frs$ is isomorphic either to $\frsl_2(\bR)$ or to $\frsu_2(\bR)$, and $\frr$ is the solvable radical of $\frg$.

If, as a module for $\frs$, $\frg$ contains four copies of the adjoint module, then $\frg$ is isomorphic to $\frs\otimes \calA$, where $\calA$ is a unital commutative associative algebra isomorphic to either $\bR[X]/(X^4)$, $\bR[X,Y]/(X^2,Y^2)$, or $\bR[X,Y]/(X^3,Y^3,X^2-Y^2)$, with the usual bracket $[x\otimes a,y\otimes b]=[x,y]\otimes ab$ for $x,y\in \frs$ and  $a,b\in\calA$.
\end{proposition}
\begin{proof}
Extending scalars to $\bC$, we obtain that the results in Proposition \ref{pr:conditions_g} hold. Then so do the arguments in Lemma \ref{le:J_T}, which are valid no matter whether $\frs\cong\frsl_2(\bR)$ or $\frs\cong\frsu_2(\bR)$. Hence the proof of Proposition \ref{pr:dimJ_4} works with one exception. In the situation in which $\calJ_0^2=\bR c$ is one-dimensional, the bilinear map $\bigl(\calJ_0/\calJ_0^2\bigr)\times\bigl(\calJ_0/\calJ_0^2\bigr)\rightarrow \calJ_0^2$ given by $(a+\calJ_0^2,b+\calJ_0^2)\mapsto ab$ is bilinear, symmetric and nonzero. Then either its rank is $1$, and we get a contradiction as in the proof of Proposition \ref{pr:dimJ_4}, or its rank is $2$ and the signature is $(1,1)$, and we argue as in the proof of Proposition \ref{pr:dimJ_4}, or its rank is $2$ and the signature is $(2,0)$ (or $(0,2)$). In this `new' case there are elements $a,b\in\calJ_0\setminus\calJ_0^2$ with $a^2=b^2\ne 0$ and $ab=0$. Then we may take $c=a^2=b^2$, and $\calJ$ is isomorphic to $\bR[X,Y]/(X^3,Y^3,X^2-Y^2)$. Besides, the bilinear form $B_\calJ$ satisfies $B_\calJ(a,b)=B_\calJ(1,ab)=0$. Also, $B_\calJ(a,a)=B_\calJ(1,a^2)=B_\calJ(1,c)=B_\calJ(b,b)$. That such a nondegenerate associative symmetric bilinear form exists is clear.
\end{proof}

Example \ref{ex:V2V1V1} works over $\bR$. We need another similar example for $\frsu_2(\bR)$.

\begin{example}\label{ex:real_W2V1}
Let $W$ be the natural module (dimension $4$) for the Lie algebra $\frsu_2(\bR)\cong\frsu(W,h)$, endowed with the symmetric nondegenerate bilinear form $B_W$ in \eqref{eq:B_W}. On the vector space $\check\fra=\frsu(W,h)\oplus W$ consider the symmetric nondegenerate bilinear form $B_{\check\fra}^\epsilon$, where $\epsilon$ is either $1$ or $-1$, such that $B_{\check\fra}^\epsilon\bigl(\frsu(W,h),W\bigr)=0$, $B_{\check\fra}^\epsilon(x,y)=\trace(xy)$ and $B_{\check\fra}^\epsilon(u,v)=\epsilon B_W(u,v)$, for $x,y\in\frsu(W,h)$ and $u,v\in W$. Consider $\check\fra$ as an abelian Lie algebra.

The linear map $\varphi:\frsu(W,h)\rightarrow \Der(\check\fra,B_{\check\fra}^\epsilon)$ given by $\varphi(x)(y)=[x,y]$ and $\varphi(x)(u)=x(u)$, for $x,y\in\frsu(W,h)$ and $u\in W$, is a Lie algebra homomorphism. The double extension $\check\frg^\epsilon=T_\varphi\bigl(\check\fra,B_{\check\fra}^\epsilon,\frsu(W,h)\bigr)$ is an irreducible quadratic Lie algebra, with a Levi subalgebra isomorphic to $\frsu_2(\bR)$ and such that, as a module for this subalgebra, it decomposes as the direct sum of three copies of the adjoint module, and one copy of the natural four-dimensional module. Its dimension is $13$. \qed
\end{example}

Alternatively, the algebra $\check\frg^\epsilon$ in Example \ref{ex:real_W2V1} can be defined on the vector space:
\[
\tilde\frg^\epsilon=\bigl(\frsu(W,h)\otimes_\bR\calA\bigr)\oplus W,
\]
with $\calA=\bR[X]/(X^3)=\bR[a]$, with $a^3=0$ ($a$ is the class of $X$ modulo $(X^3)$) and product determined by the natural Lie bracket in $\frsu(W,h)\otimes_\bR\calA$ (a scalar extension of $\frsu(W,h)$) and by
\[
\begin{split}
&[\frsu(W,h)\otimes_\bR\calA_0,W]=0,\\
 &\qquad\qquad\text{($\calA_0=\bR a+\bR a^2$ is the nilpotent radical of $\calA$)}\\[2pt]
&[x\otimes 1,w]=x(w),\\[2pt]
&[u,v]=\epsilon\Gamma(u,v)\otimes a^2,
\end{split}
\]
for $x\in\frsu(W,h)$ and $u,v,w\in W$. Indeed, the following linear map
\[
\begin{split}
\check\frg^\epsilon\qquad&\longrightarrow\qquad \tilde\frg^\epsilon\\
x+(y+w)+f&\mapsto \bigl(x\otimes 1+y\otimes a+z\otimes a^2\bigr)+w,
\end{split}
\]
for $x,y\in\frsu(W,h)$, $w\in W$ and $f\in\frsu(W,h)^*$, where $z\in\frsu(W,h)$ is such that $\trace(zs)=f(s)$ for any $s\in\frsu(W,h)$, is an isomorphism. The proof is straightforward using \eqref{eq:Gamma_B_W}.

\begin{lemma}\label{le:gepsilon}
The Lie algebras $\check\frg^1$ and $\check\frg^{-1}$ in Example \ref{ex:real_W2V1} are not isomorphic.
\end{lemma}
\begin{proof}
We must prove that $\tilde\frg^1$ and $\tilde\frg^{-1}$ are not isomorphic. To do that, we will prove that the signature of any symmetric nondegenerate invariant bilinear form on $\tilde\frg^1$ is always different from the signature of any such form on $\tilde\frg^{-1}$.

Let $B$ be any symmetric nondegenerate invariant bilinear form on $\tilde\frg^\epsilon$ ($\epsilon=\pm 1$). Since $B$ is $\frsu(W,h)\otimes 1$-invariant, the subspaces $\frsu(W,h)\otimes_\bR\calA$ and $W$ are orthogonal relative to $B$ and also we get:
\begin{itemize}
\item The restriction of $B$ to $W$ is a scalar multiple of $B_W$: $B\vert_W=\mu B_W$, for some $0\ne \mu\in\bR$.
\item There is an associative symmetric nondegenerate bilinear form $B_\calA$ on $\calA$ such that for any $x,y\in \frsu(W,h)$ and $b,c\in\calA$,
    \[
    B(x\otimes b,y\otimes c)=\trace(xy)B_\calA(b,c).
    \]
\end{itemize}
Write $B_\calA(1,a^2)=\alpha\,(\in\bR)$. Since $B_\calA$ is associative, we have $B_\calA(a,a)=B_\calA(1,a^2)=\alpha$ and $B_\calA(a,a^2)=B_\calA(1,a^3)=0$, $B_\calA(a^2,a^2)=B_\calA(1,a^4)=0$. Hence $\alpha\ne 0$, as $B_\calA$ is nondegenerate. Hence the signature of $B_\calA$ is $(1,2)$ if $\alpha>0$ and $(2,1)$ if $\alpha <0$.

For $x\in\frsu(W,h)$ and $u,v\in W$, $B(x\otimes 1,[u,v])=B(x(u),v)$, but
\[
\begin{split}
B(x\otimes 1,[u,v])&=B(x\otimes 1,\epsilon\Gamma(u,v)\otimes a^2)
    =\epsilon\trace(x\Gamma(u,v))\alpha\\
    &=\epsilon\alpha B_W(x(u),v)\quad \text{(see \eqref{eq:Gamma_B_W})}\\[4pt]
B([x\otimes 1,u],v)&=B(x(u),v)=\mu B_W(x(u),v),
\end{split}
\]
so $\epsilon\alpha=\mu$. Note that the bilinear forms $\trace(xy)$ on $\frsu(W,h)$ and $B_W$ on $W$ are positive definite. Hence, if $\epsilon=1$, the signature of $B$ is $(7,6)$ if $\alpha>0$ and $(6,7)$ if $\alpha <0$; while if $\epsilon=-1$, the signature of $B$ is $(3,10)$ if $\alpha>0$ and $(10,3)$ if $\alpha<0$.
\end{proof}

The result corresponding in the real case to Proposition \ref{pr:dimJ_3} is the following:

\begin{proposition}\label{pr:real_dimJ_3}
Let $({\frg}= \frs\oplus {\frr},B)$ be a non solvable and non simple irreducible quadratic Lie algebra over $\bR$ of dimension $\leq 13$, where $\frs$ is isomorphic either to $\frsl_2(\bR)$ or to $\frsu_2(\bR)$, and $\frr$ is the solvable radical of $\frg$.

If, as a module for $\frs$, $\frg$ contains three copies of the adjoint module, then $\frg$ is either isomorphic to $\frs\otimes \calA$, where $\calA=\bR[X]/(X^3)$ or one of the following possibilities occur:
\begin{itemize}
\item $\frs$ is isomorphic to $\frsl_2(\bR)$ and $\frg$ is isomorphic to the Lie algebra constructed in Example \ref{ex:V2V1V1} over $\bR$.
\item $\frs$ is isomorphic to $\frsu_2(\bR)$ and $\frg$ is isomorphic to either $\check\frg^1$ or $\check\frg^{-1}$, defined in Example \ref{ex:real_W2V1}.
\end{itemize}
\end{proposition}
\begin{proof}
In case $\frs$ is isomorphic to $\frsl_2(\bR)$, the arguments in the proof of Proposition \ref{pr:dimJ_3} work. Otherwise, $\frs$ is isomorphic to $\frsu_2(\bR)$. Extending scalars one gets, as in the proof of Proposition \ref{pr:dimJ_3} that $\frg$ can be written as in equations \eqref{eq:gsuWd} or \eqref{eq:gsuWWd} with $\calJ$ isomorphic to $\bR[X]/(X^3)=\bR[a]$, $\frd=0$, and with $[\frsu(W,h)\otimes\calJ_0,W]=0$ in the second case. Besides, since the only $\frsu(W,h)$-invariant bilinear form $W\times W\rightarrow \frsu(W,h)$ is, up to scalars, the bilinear map $\Gamma$ in \eqref{eq:Gamma}, we may scale and assume that in \eqref{eq:gsuWWd}, $[u,v]=\epsilon\Gamma(u,v)\otimes a^2$ for any $u,v\in W$, with $\epsilon=\pm 1$. Hence we get precisely the algebras $\tilde\frg^\epsilon\cong\check\frg^\epsilon$, $\epsilon=\pm 1$. This completes the proof.
\end{proof}

Again, Examples \ref{ex:V1V1}, \ref{ex:V1V1ddiagonal}, \ref{ex:V1V1dnilpotent} and \ref{ex:V1V1d} make sense over $\bR$ in the split case. But we need more examples.

\begin{example}\label{ex:real_W}
Let $W$ be the natural (four-dimensional) module for $\frsu_2(\bR)\cong\frsu(W,h)$, endowed with the invariant bilinear form $B_W$ in \eqref{eq:B_W}. Consider $W$ as an abelian Lie algebra. The linear map $\varphi:\frsu(W,h)\rightarrow \Der(W,B_W)$ given by the natural action: $\varphi(x)(w)=x(w)$ for any $x\in\frsu(W,h)$, $w\in W$, is a Lie algebra homomorphism. The double extension $T_\varphi(W,B_W,\frsu(W,h))$ is an irreducible quadratic Lie algebra, with a Levi subalgebra isomorphic to $\frsu(W,h)$ and such that, as a module for this subalgebra, it decomposes as the direct sum of two copies of the adjoint module and one copy of the natural module. Its dimension is $10$. \qed
\end{example}

\begin{example}\label{ex:real_dnotdiagonal}
Let $V$ be a two-dimensional real vector space endowed with a nondegenerate skew-symmetric bilinear form $(.\vert .)$.  Consider the abelian Lie algebra $\frb=V\otimes V$ with bilinear form $B_\frb$ as in Example \ref{ex:V1V1}. Fix a basis $\{u,v\}$ of $V$ with $(u\vert v)=1$, and the one-dimensional Lie algebra $\bF d$. Let $\varphi:\bF d\rightarrow \Der(\frb,B_\frb)$ be the Lie algebra homomorphism such that $\varphi(d)(u_1\otimes u)=u_1\otimes v$, $\varphi(d)(u_1\otimes v)=-u_1\otimes u$, for any $u_1\in V$. Then the double extension $\hat\frb=T_\varphi(\frb,B_\frb,\bF d)=\bF d\ltimes_{\bar\varphi}(\frb\oplus \bF d^*)$ ($d^*\in (\bF d)^*$ with $d^*(d)=1$)
is a solvable Lie algebra of dimension $6$, endowed with an invariant scalar product $B_{\hat\frb}$ such that $B_{\hat\frb}(\bF d+ \bF d^*,\frb)=0$, $B_{\hat\frb}\vert_\frb =B_\frb$ and $B_{\hat\frb}(d,d^*)=1$, $B_{\hat\frb}(d,d)=0=B_{\hat\frb}(d^*,d^*)$.

Consider now the Lie algebra homomorphism $\phi:\frsl(V)\rightarrow \Der(\hat\frb,B_{\hat\frb})$ given by $\phi(s)(\bF d+\bF d^*)=0$, $\phi(s)(u_1\otimes u_2)=s(u_1)\otimes u_2$ for any $s\in \frsl(V)$ and $u_1,u_2\in V$. The double extension $T_{\phi}(\hat\frb,B_{\hat\frb},\frsl(V))$ is an irreducible quadratic Lie algebra, with a Levi subalgebra isomorphic to $\frsl_2(\bR)$ and such that, as a module for this subalgebra, it decomposes as the direct sum of two copies of the adjoint module, two copies of the natural two-dimensional module, and two copies of the trivial one-dimensional module. Moreover, the subalgebra formed by the two copies of the trivial one-dimensional module do not act diagonally nor nilpotently on the sum of the two natural modules. \qed
\end{example}

\begin{example}\label{ex:real_dregular}
Let $W$ be the natural (four-dimensional) module for $\frsu_2(\bR)\cong\frsu(W,h)$, endowed with the invariant bilinear form $B_W$ in \eqref{eq:B_W}. Recall that $W$ is a two-dimensional complex vector space endowed with an hermitian form $h$. Consider $W$ as an abelian Lie algebra.
Let $\varphi:\bR d\rightarrow \Der(W,B_W)$ be the Lie algebra homomorphism such that $\varphi(d)(w)=\bi w$ ($\bi$ denotes the imaginary unit in $\bC$). Then the double extension $\hat W=T_\varphi(W,B_W,\bF d)=\bR d\ltimes_{\bar\varphi}(W\oplus \bR d^*)$ ($d^*\in (\bR d)^*$ with $d^*(d)=1$)
is a solvable Lie algebra of dimension $6$, endowed with an invariant scalar product $B_{\hat W}$ such that $B_{\hat W}(\bR d+ \bR d^*,W)=0$, $B_{\hat W}\vert_\frb =B_W$ and $B_{\hat W}(d,d^*)=1$, $B_{\hat W}(d,d)=0=B_{\hat W}(d^*,d^*)$.

Consider now the homomorphism $\phi:\frsu(W,h)\rightarrow \Der(\hat W,B_{\hat W})$ given by $\phi(x)(\bR d+\bR d^*)=0$, $\phi(x)(w)=x(w)$ for any $x\in \frsu(W,h)$ and $w\in W$. The double extension $T_{\phi}(\hat W,B_{\hat W},\frsu(W,h))$ is an irreducible quadratic Lie algebra, with a Levi subalgebra isomorphic to $\frsu_2(\bR)$ and such that, as a module for this subalgebra, it decomposes as the direct sum of two copies of the adjoint module, one copy of the natural four-dimensional module, and two copies of the trivial one-dimensional module. \qed
\end{example}

\begin{proposition}\label{pr:real_dimJ_2}
Let $({\frg}= \frs\oplus {\frr},B)$ be a non solvable and non simple irreducible quadratic real Lie algebra of dimension $\leq 13$, where $\frs$ is isomorphic to $\frsl_2(\bR)$ or $\frsu_2(\bR)$ and $\frr$ is the solvable radical of $\frg$, with $\frr\ne \frr^\perp$.

If, as a module for $\frs$, $\frg$ contains two copies of the adjoint module, then one of the following possibilities hold:
\begin{itemize}
\item If $\frs$ is isomorphic to $\frsl_2(\bR)$, then $\frg$ is isomorphic to the Lie algebra in one of the Examples \ref{ex:V1V1}, \ref{ex:V1V1ddiagonal}, \ref{ex:real_dnotdiagonal}, \ref{ex:V1V1dnilpotent} or \ref{ex:V1V1d}.
\item If $\frs$ is isomorphic to $\frsu_2(\bR)$, then $\frg$ is isomorphic to the Lie algebra in one of the Examples \ref{ex:real_W} or \ref{ex:real_dregular}
\end{itemize}
\end{proposition}

\begin{proof}
Extending scalars to $\bC$, the arguments in the proof of Proposition \ref{pr:dimJ_2} show that either $\frs$ is isomorphic to $\frsl_2(\bR)\cong\frsl(V)$, for a two-dimensional vector space $V$, and we are in the situation of \eqref{eq:gslVVd} with $\dim\calJ=2=\dim\calT$, with $\calJ$ isomorphic to $\bR[X]/(X^2)$, and $\frd$ abelian with $\dim\frd\leq 2$; or $\frs$ is isomorphic to $\frsu_2(\bR)\cong\frsu(W,h)$ and we are in the situation of \eqref{eq:gsuWd} or \eqref{eq:gsuWWd}, with $\calJ\cong\bR[X]/(X^2)$ and $\frd$ abelian with $\dim\frd\leq 2$.

If $\frs$ is isomorphic to $\frsl_2(\bR)$, the proof of Proposition \ref{pr:dimJ_2} works and either $\frd=0$, thus obtaining that $\frg$ is isomorphic to the Lie algebra in Example \ref{ex:V1V1}, or there is an element $d\in\frd$ such that $d(\calT)\ne 0$. Now, if the action of $d$ on $\calT$ is nilpotent, we get that $\frg$ is isomorphic either to the Lie algebra in Example \ref{ex:V1V1dnilpotent} or in Example \ref{ex:V1V1d} as in the proof of Proposition \ref{pr:dimJ_2} (with $\bF=\bR$). However, if $d$ acts faithfully on $\calT$, either $d$ acts diagonally and we get, up to isomorphism, the Lie algebra in Example \ref{ex:V1V1ddiagonal}, or we may scale $d$ and take a basis $\{x,y\}\in\calT$ with $B_\calT(x,y)=1$, $d(x)=y$ and $d(y)=-x$. Then $[d,d_{x,x}]=d_{d(x),x}+d_{x,d(x)}=2d_{x,y}$, so $d_{x,y}=0$ as $\frd$ is abelian. Also, $[d,d_{x,y}]=d_{d(x),y}+d_{x,d(y)}=d_{y,y}-d_{x,x}$, so $d_{x,x}=d_{y,y}$. But Proposition \ref{pr:conditions_g} shows that $d_{x,x}(y)=d_{y,x}(x)=d_{x,y}(x)=0$, $d_{x,x}(x)=d_{y,y}(x)=d_{x,y}(x)=0$, and hence $d_{x,x}(\calT)=0$. As in the proof of Proposition \ref{pr:dimJ_2} we have $d_{x,x}\ne 0$, and $\bR d_{x,x}$ is the kernel of the homomorphism $\frd\rightarrow \frsp(\calT,B_\calT)$. It follows that $\frd=\bR d+\bR d_{x,x}$, and the Lie bracket of $\frg$ is completely determined, so that $\frg$ is isomorphic to the Lie algebra in Example \ref{ex:real_dnotdiagonal}.

Hence, assume for the rest of this proof that $\frs$ is isomorphic to $\frsu_2(\bR)$. If $\frd=0$, then $\frg$ is a double extension of the abelian Lie algebra $\frr/\frr^\perp\simeq W$ by $\frsu(W,h)$, and by Remark \ref{re:double_extension} there is a unique possibility, up to isomorphism, so we obtain the Lie algebra in Example \ref{ex:real_W}. If $\frd\ne 0$, our Lie algebra $\frg$ can be written as in \eqref{eq:gsuWWd}, and as in Proposition \ref{pr:dimJ_2} there is an element $d\in \frd$ such that $[d,W]\ne 0$. But the action of $d$ commutes with the action of $\frsu(W,h)\otimes 1$. The centralizer of the action of $\frsu(W,h)$ on $W$ is isomorphic to the algebra of quaternions, and the action of $d$ leaves invariant the bilinear form $B_W$ so its trace is $0$. Hence we may scale $d$ and assume that the image $\hat d$ of $d$ in $\End(W)$ satisfies $\hat d^2=-\id$. Therefore, we may assume without loss of generality that the action of $d$ is the multiplication on $W$ by the imaginary unit $\bi$. In particular, $d$ acts faithfully on $W$. By extending scalars to $\bC$, the proof of Proposition \ref{pr:dimJ_2} shows that $\frd$ is a two-dimensional abelian Lie algebra, and the kernel of the action of $\frd$ on $W$ is one-dimensional. We conclude that the Lie bracket is completely determined and hence our Lie algebra $\frg$ is isomorphic to the Lie algebra in Example \ref{ex:real_dregular}.
\end{proof}

\smallskip

Our last result, obtained putting together all the previous work, and taking into account the simple Lie algebras over $\bR$ of dimension $\geq 13$, summarizes the classification in the real case:

\begin{theorem}
The complete list, up to isomorphisms, of the non solvable irreducible quadratic real Lie algebras $\frg$ with $\dim\frg\leq 13$ is the following:
\begin{enumerate}
\item $\dim\frg=3$: the simple Lie algebras $\frsl_2(\bR)$ and $\frsu_2(\bR)$,
\item $\dim\frg=6$: the real simple Lie algebra $\frsl_2(\bC)$, and the trivial $T^*$-extensions $T_0^\star(\frsl_2(\bR))$ and $T_0^\star(\frsu_2(\bR))$,
\item $\dim\frg=8$: the simple Lie algebras $\frsl_3(\bR)$, $\frsu_3(\bR)$ and $\frsu_{2,1}(\bR)$,
\item $\dim\frg=9$: the `scalar extensions' $\frsl_2(\bR)\otimes_\bR \bR[X]/(X^3)$ and $\frsu_2(\bR)\otimes_\bR \bR[X]/(X^3)$,
\item $\dim\frg=10$: the simple Lie algebras $\frso_5(\bR)$, $\frso_{4,1}(\bR)$ and $\frso_{3,2}(\bR)$, and the Lie algebras in Examples \ref{ex:V1V1} and \ref{ex:real_W},
\item $\dim\frg=11$: the double extensions $\frd(4)$ and $\check\frd(4)$, and the Lie algebra in Example \ref{ex:V1V1d},
\item $\dim\frg=12$: the `scalar extensions' $\frs\otimes\bR\calA$, with either $\frs=\frsl_2(\bR)$ or $\frsu_2(\bR)$ and with $\calA$ equal to $\bR[X]/(X^4)$, $\bR[X,Y]/(X^2,Y^2)$ or $\bR[X,Y]/(X^3,Y^3,X^2-Y^2)$, the Lie algebras in Examples \ref{ex:V1V1ddiagonal}, \ref{ex:V1V1dnilpotent}, \ref{ex:real_dnotdiagonal}, \ref{ex:real_dregular}, and the trivial $T^\star$-extension $T_0^\star(\frsl_3(\bC))$.
\item $\dim\frg=13$: the double extensions $\frd(6)$ and $\check\frd(6)$,  and the Lie algebras in Examples \ref{ex:V2V1V1} and \ref{ex:real_W2V1}.
\end{enumerate}
\end{theorem}

\medskip

\section{Conclusion and open question}

By means of our classification above, we can give easily  the list (up isomorphism) of the perfect quadratic Lie algebras of dimension $\leq 13$. Moreover, our classification reduces the classification of the not necessarily irreducible, non solvable, quadratic Lie algebras of dimension $\leq 13$ to the classification of the solvable quadratic Lie algebras of dimension $\leq 7$. Recall that the classification of the nilpotent quadratic Lie algebras of dimension $\leq 7$ is obtained in \cite{FavS},
but the classification of the solvable non nilpotent quadratic Lie algebras of dimension $\leq 7$ is still an open question.


\end{document}